\documentclass[12pt]{article}

\usepackage{tikz-cd}
\usetikzlibrary{graphs,decorations.pathmorphing,decorations.markings}
\usetikzlibrary{arrows}

\tikzset{
	negated/.style={
		decoration={
			markings,
			mark=at position 0.5 with {\node[transform shape] (tempnode) {{\small /}};},
		},
		postaction={decorate},
	},
}

\usepackage{amsmath}
\usepackage{latexsym}
\usepackage{amsfonts}
\usepackage{amssymb}
\usepackage{amscd}
\usepackage{theorem}
\usepackage{a4wide}
\usepackage{color}

\newtheorem{theorem}{Theorem}
\newtheorem{definition}[theorem]{Definition}
\newtheorem{lemma}[theorem]{Lemma}
\newtheorem{proposition}[theorem]{Proposition}
\newtheorem{corollary}[theorem]{Corollary}

{\theorembodyfont{\rmfamily}
  \newtheorem{example}[theorem]{Example}

  \newtheorem{remark}[theorem]{Remark}

}

\newenvironment{proof}{    
  \noindent
  \textbf{Proof.}}{
  \hfill $\Box$
  \vspace{3mm}
}

\numberwithin{equation}{section}

\newcommand{\N}{\mathbb{N}} 
\newcommand{\Z}{\mathbb{Z}} 
\newcommand{\R}{\mathbb{R}} 
\newcommand{\C}{\mathbb{C}} 
\newcommand{\D}{\mathbb{D}} 

\newcommand{\W}{C_{w,\varphi}} 

\title{Dynamics of weighted composition operators on spaces of continuous functions}
\author{M.J. Beltr\'an-Meneu, E. Jord\'a and M. Murillo-Arcila}

\begin{document}
\maketitle

\begin{abstract}

Our study is focused on the dynamics of weighted composition operators  defined on a locally convex space $E\hookrightarrow (C(X),\tau_p)$  with $X$ being a topological  Hausdorff space containing at least two different points and such that the evaluations $\{\delta_x:\ x\in X\}$ are linearly independent in $E'$.  We prove, when $X$ is compact and $E$ is a Banach space  containing a nowhere vanishing function, that a weighted composition operator $\W$ is never weakly supercyclic on $E$. 
We also prove that if the symbol $\varphi$ lies  in the unit ball of $A(\D)$, then every weighted composition operator can never be $\tau_p$-supercyclic neither on  $C(\D)$ nor on the disc algebra $A(\D)$.
Finally,  we obtain  Ansari-Bourdon type  results and conditions on the spectrum for arbitrary weakly supercyclic operators, and we provide necessary conditions for a composition operator to be weakly supercyclic on the space of holomorphic functions defined in non necessarily simply connected planar domains. As a consequence, we show that no composition operator can be weakly supercyclic neither on the space of holomorphic functions on the punctured disc nor in the punctured plane.\\
\textbf{Keywords:} weighted composition operator, weak supercyclicity, disc algebra,  space of holomorphic functions\\
\textbf{Mathematics Subject Classification (2010):} 47A16, 47B33, 46E15
\end{abstract}

\section{Introduction}
Troughout this paper, let $X$ denote  a topological Hausdorff space containing at least two different points and let $C(X)$ be the space of continuous functions on $X.$   On this space, consider the weak topology, the pointwise convergence topology $\tau_p,$ and whenever $X$ is compact, the supremum norm topology.
Let $E$  be a locally convex space continuously included in  $(C(X), \tau_p)$ and such that the set $\{\delta_x:  x\in X\}$ is linearly independent in $E',$ where $\delta_{x}$ is the functional on $E$ of point evaluation at $x.$
The aim of this paper is to investigate weak forms of supercyclicity of the \emph{weighted composition operator} $\W$ when it is well defined in $E$. We refer to the next section for the precise notation and definitions.


Given an operator $T$ defined on a topological vector space $(F, \tau)$,  for every $n\in \N$ the operator $T^n: F\rightarrow F$ is defined as the $n$-th iterate of $T,$ i.e., $T^n:=T\circ\overset{n)}{\cdots}\circ T,$ and $T^0=I.$  A point $f\in F$ is said to be \emph{periodic} if there exists $n\in \N$ such that $T^nf=f,$ and it is a \emph{fixed point} if $Tf=f.$  We say that an operator $T$  is  $\tau$-\textit{hypercyclic} if there is some vector $f\in F$ (called $\tau$-hypercyclic vector) whose orbit $\text{Orb}(T, f):=\{T^nf: \ n=0,1,\dots \}$ is dense in $(F, \tau).$  If $\text{Orb}(T, \text{span}\{f\})=\{\lambda T^nf: \lambda \in \C,  n=0,1,\dots \}$ is dense, we say that $T$ is $\tau$-\textit{supercyclic}, and if $\text{span} \{\text{Orb}(T, f)\}=\text{span}\{T^nf: \ n=0,1,\dots \}$ is dense,  it is said to be $\tau$-\textit{cyclic}. In the case $\tau$ denotes the weak topology, the operator is said to be weakly hypercyclic (resp. weakly supercyclic or weakly cyclic), and if $\tau=\tau_p,$ the operator is said to be pointwise hypercyclic (resp. pointwise supercyclic or pointwise cyclic).
If $F$ is a separable infinite dimensional Banach or Fr\'echet space and $\tau$ denotes the strong topology, the operator is said to be hypercyclic (resp. supercyclic or cyclic). It is clear that for any of the aforementioned topologies, $\tau$-hypercyclicity implies $\tau$-supercyclicity, which in turn implies $\tau$-cyclicity.  In the case of Banach spaces, if we mix the norm and weak topologies 	  we first point out that weak cyclicity is equivalent to cyclicity in the norm topology because the weak closure of the convex set span$\{\text{Orb}(T, f)\}$ coincides with the norm closure. However, weakly hypercyclic operators are not always norm hypercyclic as it was shown by Chan and Sanders in \cite[Corollary 3.3]{ChanSanders} and  weakly supercyclic operators are not necessarily norm supercyclic \cite[Theorem 2.3]{Sanders 2004}.
In the setting of Banach spaces,  norm and weakly supercyclic operators share many  properties such  as the density of supercyclic vectors. It is well known that if $T$ is norm supercyclic then the set of all norm supercyclic vectors for $T$ is norm dense in $F$ and Sanders \cite{Sanders 2004} proved that if $T$ is weakly supercyclic,  the set of all weakly supercyclic vectors for $T$ is also norm dense in $F.$

For a good exposition of the subject of linear dynamics we refer the reader to the monographs by Bayart and Matheron \cite{BayartMatheron} and by Grosse-Erdmann and Peris \cite{GE_Peris}, and concerning composition operators we refer to the books by Cowen and MacLuer \cite{Cowen_MacCluer_book} and by Shapiro \cite{Shapiro1993}.



 Ansari and Bourdon \cite{AnsariBourdon} showed that an isometry on an infinite dimensional
Banach space cannot be norm supercyclic. However, surjective  isometries  can be weakly supercyclic. Sanders proved in \cite[Theorem 2]{Sanders 2005} that the bilateral backward shift is weakly supercyclic on $c_0(\Z)$. Shkarin improved this theorem showing that in fact the bilateral  backward shift operator is weakly supercyclic on $\ell_p(\Z), \ p > 2$  \cite[Theorem 1.5]{Shkarin}. For $1\leq p\leq  2,$  Montes-Rodr\'{\i}guez and Shkarin  \cite[Theorem 6.3]{MontesShkarin} and Shkarin \cite[Theorem 1.5]{Shkarin} showed   that for weighted bilateral shifts, weak supercyclicity is equivalent to supercyclicity on $\ell_p(\Z)$. From this together with Ansari-Bourdon's theorem, it immediately follows that an isometric weighted bilateral shift cannot be weakly supercyclic  on $\ell_p(\Z)$ for $1\leq p\leq  2 $. See also \cite[page 253]{BayartMatheron}. Besides other results about weak forms of supercyclicity, Bayart and Matheron exhibit in \cite{BM} an example of a unitary operator on a Hilbert space which is weakly supercyclic.


From this perspective, we study supercyclity on spaces of functions endowed with the weak topology and also with the pointwise convergence topology, which is in general the weakest natural locally convex Hausdorff topology. We discuss here the difference between both concepts showing that they are in general different by giving an example of a $\tau_p$-supercyclic operator which is not cyclic, hence, not (weakly) supercyclic.  For surjective isometries we get some extensions of Ansari-Bourdon's theorem for important spaces of functions. For $C(X)$ when $X$ is compact, we get that a surjective isometry (which is always a weighted composition operator by the Banach-Stone theorem) is never weakly supercyclic, and for the disc algebra $A(\D)$ we go further. No weighted composition operator on $A(\D)$ is pointwise supercyclic.

These results on weighted composition operators on $A(\D)$ connect with recent research about weak and strong supercyclicity  of  weighted composition operators on spaces of holomorphic functions. Yousefi and Rezaei \cite{YousefiRezaei} and Kamali et al. \cite{Kamalietal2010} investigated the hypercyclicity and supercyclicity of $\W$ on the space of holomorphic functions on the disc $H(\D)$, both with respect to the compact-open topology and to its corresponding weak topology. In \cite{Bes2014} B\`es proved, among other results, that weak supercyclicity and the topologically mixing property (a dynamical concept more restrictive than hypercyclicity) are equivalent notions for $\W$ on the space of holomorphic functions on a simply connected plane domain, which in turn are satisfied if and only if the weight $w$ is zero-free and the  symbol $\varphi$ is univalent and without fixed points. More recently, in \cite{Moradietal2017} Moradi et al. proved that a class of weighted composition operators which contains every composition operator on  some Banach spaces of analytic functions such as the disc algebra and  the analytic Lipschitz space does not contain weakly supercyclic operators.  For supercyclicity and weak supercyclicity of operators defined on spaces of functions with real variable, see \cite{BMS,FGJ,MRS} and the references therein.

There is also active research focused on the connections between spectral theory of linear operators   defined on Banach or locally convex spaces, the linear dynamics of the operator $T$ and the dynamics of its adjoint $T'$.  Specially relevant are the connections between the linear dynamics of $T$  and the point spectrum $\sigma_p(T')$  of its adjoint. Herrero proved in \cite{Herrero} that the point spectrum of a supercyclic operator defined on a Hilbert space consists at most of one point, and whenever it is not empty, the dimension of the subspace formed by the eigenvectors is one. Peris extended this result  in \cite{Peris2001} to supercyclic operators defined on locally convex spaces.  Ansari proved in \cite{Ansari1995} that if $T$ is cyclic and  the interior of $\sigma_p(T^*)$ is empty, then $T$ has a norm dense collection of cyclic vectors. The Ansari-Bourdon theorem mentioned above about the non supercyclicity of isometries on Banach spaces is a consequence of a theorem which asserts that for a supercyclic power bounded operator $T$  defined on a Banach space, the orbits of the adjoint $T^*$  are everywhere $\omega^*$ convergent to 0.  For recent research extending these classic results we refer to \cite{AlbaneseJornet2018,AlemanSuciu,Dugal}. In the last section of the paper we use our results about weighted composition operators  to get results for arbitrary  operators defined on locally convex spaces, obtaining  conditions  in the dynamics of the adjoint $T'$ which are necessary for $T$  being weakly supercyclic. In case of  operators defined on Banach spaces, the necessary conditions are related to the  point spectrum of the operator.

	In \cite{BonetPeris}, Bonet and Peris proved that every separable infinite dimensional Fr\'echet space admits a hypercyclic surjective operator. Grosse-Erdmann and Mortini showed in \cite[Theorem 3.21]{GE_Mortini}  that if $U\subseteq \C$ is a non simply connected domain such that $\widehat{\C}\setminus U$  has finitely many  bounded components then $H(U)$ does not support any hypercyclic composition operator. As an application of our results we provide an example of a space of holomorphic functions endowed with the compact open topology which admits no weakly supercyclic composition operators. Concretely, we show that every composition operator on $H(\D\setminus\{0\})$ or on $H(\C\setminus\{0\})$ is never weakly supercyclic. This result is strongly connected with the open question proposed in  \cite[Problem 3]{Bes2014}. 

\subsection{Notation  and Outline of the Paper}

Our notation  is standard. We stand $E'$ for the dual of a locally convex space unless $E$ is a Banach space, in which case its dual is denoted   by $E^*.$  The adjoint of a continuous linear mapping $T\in L(E)$ is denoted by $T'$ in the general case and by $T^*$ when $E$ is Banach.  As mentioned in the introduction, our study is focused on dynamics of weighted composition operators  defined on an infinite dimensional locally convex space $E\hookrightarrow (C(X),\tau_p)$ with $X$ being a topological  Hausdorff space  such that the evaluations $\{\delta_x: x\in X\}$ are linearly independent in $E'$. We consider a continuous function $w:X\to \C$  (the multiplier) and a continuous $\varphi: X\to X$  (the symbol) such that the weighted composition operator $\W: E\to E$, $f\mapsto w (f\circ \varphi)$ is well defined and continuous. The operator $\W$ combines the classical composition operator
$C_{\varphi}:E\rightarrow E, \  f \mapsto f \circ \varphi$ with the pointwise multiplication operator $M_{w}:E\rightarrow E,\  f \mapsto w\cdot f.$

  In \cite{Moradietal2017}, Moradi et al. provide sufficient conditions under which
a weighted composition operator on a Banach space of analytic functions
is not weakly supercyclic, and they proved that for some Banach spaces  $Y$ of analytic functions on $\D,$ the unweighted composition operator $C_{\varphi}$  is not weakly supercyclic. These spaces satisfy that every element in $Y$ has a continuous extension to the closed unit disc $\overline{\D}$ and for every $z$ in the boundary $\partial{\D},$ $\delta_z$ is bounded.  The disc algebra

$$A(\D)=\{f\in H^{\infty}(\D): \ f \text{ continuous on } \overline{\D}\},$$
and the analytic Lipschitz spaces  $ Lip_{\alpha}(\D),$ $0<\alpha\leq1,$
$$ Lip_{\alpha}(\D)=\{f\  \mbox{analytic}\hspace{0.1cm}\mbox{in}\hspace{0.1cm} \D: |f(z)-f(w))|=O(|z-w|^{\alpha}) \hspace{0.1cm}\mbox{for}\hspace{0.1cm}\mbox{ all}\hspace{0.1cm} z,w \in \overline{\D}\},$$
are examples of such spaces. As a corollary of Proposition 4 in \cite{Moradietal2017} one immediately obtains the following:

\begin{corollary}\label{unbounded}
	Let $\varphi, w \in Y$ such that $\varphi(\D)\subseteq \D$ and let $a \in \overline{\D}$ be a fixed point of $\varphi$ such that $w(a)\neq 0.$ If $\W$ is weakly supercyclic on $Y,$ then the set $$\left\{\frac{\prod_{m=0}^{n}w(\varphi^m(z))}{w^{n}(a)} , \ n\in \N \right\}$$
	is unbounded for every $z\in \overline{\D}\setminus \{a\}.$
\end{corollary}

In Section \ref{sec1} we strengthen the necessary condition provided in Corollary \ref{unbounded} and we extend it to general locally convex spaces (see Proposition \ref{necessary_wf}) instead of the concrete space of analytic functions $Y$. In Theorem \ref{CXcompact} we use this result to prove the following, from which it immediately follows that a weighted composition operator can never be weakly supercyclic on $A(\D)$,  neither on a Banach space dense and continuoulsy embedded in it:\\
\newline
\textbf{Theorem A} Let $X$ be compact and $E$ a Banach space containing a nowhere vanishing function and satisfying $E\hookrightarrow (C(X), \|\ \|_{\infty}).$  A weighted composition operator $\W: E \rightarrow E$  is never weakly supercyclic.\\



We then restrict our study to the case where  $E$ is a subspace of $C(X)$ endowed with the $\tau_p$ topology. More concretely, in Theorem \ref{noweakcases} we provide sufficient conditions for the symbol $\varphi$ which ensure no $\tau_p$-supercyclicity of $\W$ such as the existence of a non-constant convergent orbit, the existence of stable orbits around a fixed point,  or, whenever  $X$ is compact, the existence of periodic (not fixed) points.  B\`es  \cite{Bes2014} gives a complete characterization of weak supercyclicity of $\W$ on $H(\D).$  In particular, he shows that  for any multiplier $w$ and any symbol $\varphi: \D\to\D$ with a fixed point the weighted composition operator $\W$ cannot be weakly supercyclic. From our Theorem \ref{noweakcases}  we show that assuming these conditions $\W$ is  not $\tau_p$-supercyclic on $C(\D)$, which easily implies B\`es' result since $H(\D)$ is densely embedded in $(C(\D),\tau_p)$.


In \cite{Moradietal2017}, Moradi et al. provide sufficient conditions under which
a weighted composition operator on certain Banach spaces of analytic functions
is not weakly supercyclic. For these spaces we prove that no further restriction in the operator than being well defined is needed to obtain that it can never be even pointwise supercyclic.  Concretely, in Theorem \ref{disk} we prove:\\

\textbf{Theorem B}
Let $E\hookrightarrow (C(\overline{\D}), \tau_p)$ be such that the evaluations on $\overline{\D}$ are linearly independent in $E'.$ If for some $\varphi \in A(\D)$  and $w\in C(\overline{\D})$  the weighted composition operator $\W: E\rightarrow E$ is well defined, then it is not $\tau_p$-supercyclic. As a consequence, any weighted composition operator $\W$ is never $\tau_p$-supercyclic on the disc algebra $A(\D)$ nor on the analytic Lipschitz spaces $ Lip_{\alpha}(\D),$ $0<\alpha\leq1.$
\\

%

In Section \ref{sec4}, we use our results on composition operators  to obtain necessary conditions for an arbitrary operator $T$ defined on a locally convex space to be weakly supercyclic.  These conditions are connected with the classical Ansari-Bourdon's theorem. More concretely, in Theorem \ref{lcs} and Corollary \ref{finalcorollary} we get:\\
\label{lcs}

\textbf{Theorem C} \begin{itemize} \item[(i)] Let $E$ be a locally convex space and $T:E\to E$  a continuous linear operator  which is weakly supercyclic and satisfies $q\circ T\leq q$  for a continuous norm $q$ on $E$. Then $\sigma_p(T)\cap \partial\D=\emptyset$ and  $\sigma_p(T')\cap \partial \D=\emptyset$. In particular, neither $T$ nor $T'$ have non zero fixed points.
\item[(ii)] If $X$ is a Banach space and $T:X\to X$ is a weakly supercyclic operator, then both the point spectrum of $T$ and that of $T^*$ are contained in the open ball $B(0,\|T\|)$.

\end{itemize}

As an application of our results we succeed in showing that every composition operator on $H(U),$ for $U$ being the punctured disc or the punctured plane, is never weakly supercyclic, which is strongly related to \cite[Problem 3]{Bes2014}:\\

\textbf{Theorem D} 
	The spaces $H(\D\setminus\{0\}))$ and $H(\C\setminus\{0\})$ admit no weakly supercyclic composition operators.\\

Finally, in the last section of our paper and motivated by the fact that in Proposition \ref{necessary_wf} only the pointwise convergence topology  is needed,  we study the connections between the concepts of weak supercyclicity and  $\tau_p$-supercyclicity on weighted composition operators.  We study the relation between these two concepts and cyclicity and we provide different examples that permit separate them, i.e. there are pointwise and cyclic supercyclic operators which are not weakly supercyclic (Example \ref{weakandpointwisedifferent}), pointwise supercyclic operators which are not cyclic  (Proposition \ref{taup_nocyclic}) and cyclic operators which are not pointwise supercyclic (Example \ref{multiplicationex}).
In the last part of this section we pay special attention to the study  of weighted composition operators defined on  $C(\overline{\D})$  and  $C(\partial\D),$ where $\partial\D$ stands for the unit circle.

From our results, we conjecture  that if  $X$ is compact,  $E\hookrightarrow (C(X),\tau_p)$ is a Banach space  and the operator $\W: E \rightarrow E$  is isometric, then $\W$  is not $\tau_p$-supercyclic. Of course, the conjecture is not true when $X$ is only assumed to be locally compact, since \cite{Sanders 2004} shows  weak supercyclicity of the backward shift on $c_0(\Z)$, which is an isometric composition operator.\\

\section{Dynamics of $\W$ on spaces of continuous functions}\label{sec1}

In this section we study weak supercyclicity of the weighted composition operator defined on an infinite dimensional locally convex space of continuous functions $E\hookrightarrow (C(X),\tau_p)$ such that all the evaluations $\{\delta_x:\ x\in X\}$ are linearly independent and we obtain important results in the setting of Banach spaces. Our first result provides necessary conditions for  $\W$ to be weakly (even pointwise) supercyclic. The proof is analogous to the one in \cite[Proposition 2.1]{Bes2014}. We include it here for the sake of completeness.

\begin{proposition}\label{necessaryCK}
If $\W:E\rightarrow E$ is $\tau_p$-supercyclic, then:
	\begin{itemize}
		\item[(i)] $w$ is zero-free,
		\item[(ii)] $\varphi$ is univalent.
	\end{itemize}
\end{proposition}

\begin{proof}
	(i) If $w(z_0)=0$ for some $z_0\in X$, then for any $f\in E$ we have $Orb(\W,\text{span}\{f\})\subseteq \text{span}\{f\}\cup Ker(\delta_{z_0})  \subsetneq  E,$ and hence is  not  dense in $E$ with respect to $\tau_p,$ a contradiction.\\
	(ii) Suppose that $\varphi(z_1)=\varphi(z_2)$ for some $z_1,z_2\in X.$ By $(i)$ 	and the fact that  $\delta_{z_1}$ and $\delta_{z_2}$ are linearly independent, there exists $g\in E$ such that $g(z_1)\neq \frac{w(z_1)}{w(z_2)}g(z_2).$ Consider $f\in E$ a $\tau_p$-supercyclic function of $\W.$ For
	$$\epsilon:=\left|g(z_1)- \frac{w(z_1)}{w(z_2)}g(z_2)\right|>0,$$
	let $\lambda\in\C$, $n\in\N$ such that
	$$|\lambda C_{w,\varphi}^nf-g|(z)<\frac{\epsilon}{4\max\{\frac{|w(z_1)|}{|w(z_2)|},1\}},\ (z=z_1,z_2).$$
	Observe that
	$$C_{w,\varphi}^nf(z_1)=C_{\varphi}^nf(z_1)\frac{w(z_1)}{w(z_2)}\prod_{j=0}^{n-1}C_{\varphi}^j(w)(z_2)=\frac{w(z_1)}{w(z_2)}C_{w,\varphi}^nf(z_2),$$
	then it follows that
	$$\epsilon=\left|g(z_1)- \frac{w(z_1)}{w(z_2)}g(z_2)\right|\leq \left|g(z_1)-\lambda C_{w,\varphi}^nf(z_1) \right|+\left| \lambda C_{w,\varphi}^nf(z_1)-\frac{w(z_1)}{w(z_2)}g(z_2)\right|\leq $$
	$$\frac{\epsilon}{4}+\left| \frac{w(z_1)}{w(z_2)}\right|\left|\lambda C_{w,\varphi}^nf(z_2)-g(z_2)\right|\leq \frac{\epsilon}{4}+ \frac{\epsilon}{4},$$
	a contradiction. 	Then $\varphi$ must be univalent.
\end{proof}

As a consequence of Proposition \ref{necessaryCK},  in what follows we only consider weighted composition operators $\W$ such that $w$ is zero-free on $X$ and $\varphi$ is univalent.

\begin{remark}\label{wfno0}
	If $\W$ is $\tau_p$-supercyclic on $E,$  then for every  $z\in X$ and every $\tau_p$-supercyclic function $f$ there exists $n\in \N$ such that $f(\varphi^n(z))\neq 0. $ Otherwise, $\text{Orb}(\W,\text{span}\{f\circ\varphi \})\subseteq  Ker(\delta_{z})\subsetneq  E,$ and hence  $\text{Orb}(\W,\text{span}\{f\circ\varphi \})$ is $\tau_p$ nowhere dense. This is a contradiction since $f\circ\varphi$ is $\tau_p$-supercyclic if $f$ is so.
\end{remark}

The next proposition strengthens the necessary condition for weak supercyclicity provided in Corollary \ref{unbounded} and extends  it to the space of continuous functions $E,$ not necessarily analytic. We remark that we use the tool employed in \cite[Proposition 1.26]{BayartMatheron} and \cite[Lemma 1]{Peris2001} to  prove that the adjoint of a supercyclic operator in a locally convex space $E$ cannot have two linearly independent eigenvectors.

\begin{proposition}\label{necessary_wf}
If $\W: E\rightarrow E$  is $\tau_p$-supercyclic, then
$$\overline{\left\{  \frac{\prod_{m=0}^{n-1}w(\varphi^m(z_1))f(\varphi^n(z_1))}{\prod_{m=0}^{n-1}w(\varphi^m(z_2))f(\varphi^n(z_2))}  , \ n\in \N :  f(\varphi^n(z_2))\neq 0  \right\}}=\mathbb{\C}\quad $$
for every $\tau_p$-supercyclic function $f$ and every $z_1\neq z_2\in X.$  \\
If in addition, $(\varphi^n(z_1))_n$ and $(\varphi^n(z_2))_n$ are convergent in $X$ to some (fixed) points $a$ and $b,$ respectively,  then
$$\overline{\left\{  \frac{\prod_{m=0}^{n}w(\varphi^m(z_1))}{\prod_{m=0}^{n}w(\varphi^m(z_2))}  , \ n\in \N \right\}}=\mathbb{\C}.\quad $$
\end{proposition}

\begin{proof}
Let $z_1,z_2\in X$ be such that $z_1\neq z_2.$ The mapping $F: E\rightarrow \C^2$ defined as $F(g)=(g(z_1),g(z_2))$ is $\tau_p$-continuous and, since $\{\delta_{z_1}, \delta_{z_2}\}$ is linearly independent in $E^{'},$ it follows that $F$  is surjective.
This implies that if  $f\in E$ is a $\tau_p$-supercyclic vector of $\W$, then the set $$\left\{\left(\lambda\prod_{m=0}^{n-1}w(\varphi^m(z_1))f(\varphi^n(z_1)), \lambda\prod_{m=0}^{n-1}w(\varphi^m(z_2)) f(\varphi^n(z_2))\right) :\lambda\in\mathbb{\C},n\in\N\right\}$$
is dense in $\mathbb{\C}^2.$  Thus, given $c\in\mathbb{\C}\backslash\{0\}$, there exists an increasing  sequence $(n_k)_k$ such that $\lambda_{n_k}\neq0$, $ f(\varphi^{n_k}(z_i))\neq 0$ for $i=1,2,$ and
$$\left(\lambda_{n_k}\prod_{m=0}^{n_k-1}w(\varphi^m(z_1))f(\varphi^{n_k}(z_1)), \lambda_{n_k}\prod_{m=0}^{n_k-1}w(\varphi^m(z_2))f(\varphi^{n_k}(z_2))\right)\rightarrow \left(c,1\right).$$
As a consequence,  $$\lim_k\frac{\prod_{m=0}^{n_k-1}w(\varphi^m(z_1))f(\varphi^{n_k}(z_1))}{\prod_{m=0}^{n_k-1}w(\varphi^m(z_2))f(\varphi^{n_k}(z_2))}=c$$ and  the first part of the  proposition holds. \\
If in addition, $(\varphi^n(z_1))_n$ and $(\varphi^n(z_2))_n$ are convergent to the fixed points $a$ and $b$ in $X,$ respectively, since $\lim_nf(\varphi^n(z_1))=f(a)\neq 0$ and $\lim_nf(\varphi^n(z_2))=f(b)\neq 0$ by Remark \ref{wfno0}, we get
$$\overline{\left\{  \frac{\prod_{m=0}^{n}w(\varphi^m(z_1))}{\prod_{m=0}^{n}w(\varphi^m(z_2))}  , \ n\in \N \right\}}=\mathbb{\C}.\quad $$
\end{proof}

\begin{remark}\label{Prop 4 suau}
	
Observe that in Proposition \ref{necessary_wf} we do not need all evaluations $\delta_z,$ $z\in X,$ to be linearly independent.  It is enough to assume that $\{\delta_{z_1}, \delta_{z_2}\}$ is linearly independent. Note that $w(z)$ can be zero  if $\delta_z=0$ in $E',$ but if $\delta_z\neq 0,$ then $w(\varphi^n(z)) \neq 0$ for every $n\in \N.$ Just consider an increasing sequence $(n_k)_k$ such that	$\lambda_{n_k}\prod_{m=0}^{n_k-1}w(\varphi^m(z))f(\varphi^{n_k}(z))$ tends to $1$.
\end{remark}

Now we present one of the main results of this section, which solves in the negative  the  problem of weak supercyclicity of the weighted composition operator on $E$ whenever  $X$ is compact.

\begin{theorem}\label{CXcompact}
	Let $X$ be compact and let $E$ be a Banach space satisfying  $E\hookrightarrow (C(X), \|\ \|_{\infty})$ and containing a nowhere vanishing function. The weighted composition operator $\W: E \rightarrow E$  is never weakly supercyclic.
\end{theorem}

\begin{proof}
	If $\W: E\rightarrow E$  is weakly supercyclic, the set of weakly supercyclic vectors is norm dense in $E,$ and thus, $\|\ \|_{\infty}$-dense \cite[Proposition 2.1]{Sanders 2004}. Then, by the hypothesis of existence of nowhere vanishing functions on $E$ we can easily get a weakly supercyclic function $f$ such that $\epsilon\leq |f(x)|$ for every $x\in X$ and for some $\epsilon>0.$ 	Since the multiplication operator $M_w: E\rightarrow E$  is not weakly supercyclic by \cite[Proposition 1.26]{BayartMatheron} because $\{\delta_x, \ x\in X\}$  is a set of  independent eigenvectors of $E^*$,  we can assume that there exists some $x_0\in X$ such that   $x_1=\varphi(x_0)\neq x_0.$  Therefore, as $X$ is compact and $w$ and $f$ are continuous and have no zeros in $X,$ there exists $C>0$ such that $$\left|\frac{\prod_{m=0}^{n-1}w(\varphi^m(x_1))f(\varphi^n(x_1))}{\prod_{m=0}^{n-1}w(\varphi^m(x_0))f(\varphi^n(x_0))}\right|=\left|\frac{w(\varphi^{n}(x_0))f(\varphi^{n+1}(x_0))}{w(x_0)f(\varphi^{n}(x_0))}\right|\leq C$$
	for every $ n\in \N,$	a contradiction by Proposition \ref{necessary_wf}.
\end{proof}


 Our main results apply to $C(X)$ when $X$ is compact, but relaxing this hypothesis and not requiring all evaluations $\delta_z,$ $z\in X,$ to be linearly independent in Theorem \ref{noweakcases},  permits us to obtain more general results which are useful in the following section. We give sufficient conditions that ensure the operator $\W$ is not  $\tau_p$-supercyclic. We first introduce the following definition.

\begin{definition}
	Let $\varphi: X\rightarrow X$ and let $z_0\in X$ be a fixed point of $\varphi.$ $\varphi$ is said to have  \emph{stable orbits around $z_0$} if there exists a fundamental family of connected compact neighbourhoods of $z_0,$ $(V_j)_j\subseteq X$  such that $\varphi(V_j)\subseteq V_j $  for every $j\in \N.$
\end{definition}    

\begin{theorem}\label{noweakcases}
Let $X$ be a topological Hausdorff space and let $E\hookrightarrow (C(X),\tau_p)$.  If any of the following conditions holds, then the operator $\W: E \rightarrow E$ is not $\tau_p$-supercyclic:
\begin{itemize}
	\item [i)] $\varphi$ has two  fixed points $\{z_1,z_2\}$ such that the evaluations $\{\delta_{z_1},\delta_{z_2}\}\subseteq E'$ are linearly independent.
	 	\item [ii)] there exists an orbit $\{\varphi^n(z_1), n=0,1, \dots\},$ $z_1\in X,$ non-constant and convergent to an element $z_0\in X$ satisfying $ \delta_{z_0}\neq 0$ and $\{\delta_{z_1}, \delta_{\varphi(z_1)}\}$ linearly independent in $E'$.
	 		\item [iii)] $X$ is compact and $\varphi$ has a periodic (not fixed) point $z_1$ satisfying  that  $\{\delta_{z_1}, \delta_{\varphi(z_1)}\}$  are linearly independent in $E'$.
	 		\item [iv)]  $X$ is compact, $\varphi$ has a fixed point $z_2$ such that $|w(z)|\leq |w(z_2)|$ for all $z\in X$ and there exists $z_1\neq z_2$ such that $\{\delta_{z_1},\delta_{z_2}\}$ are linearly independent in $E'$.
	 		\item [v)] All evaluations $\delta_z,$ $z\in X,$ are linearly independent,   $\varphi$ has a fixed point $z_0$ such that $z_0$ is an accumulation point of $X$, and $\varphi$ has stable orbits around $z_0$.
\end{itemize}
\end{theorem}

\begin{proof}
Assume $\W$ is $\tau_p$-supercyclic and  $f\in E$ is a $\tau_p$-supercyclic function.  For $z_1$ and $z_2$  as in (i)  the set $\left\{\left(\frac{w(z_1)}{w(z_2)}\right)^n\frac{f(z_1)}{f(z_2)}:n\in\N\right\}$ cannot be dense in $\C,$ since it converges to 0, diverges to infinity or lies in $r\partial \D$ for $r=\frac{f(z_1)}{f(z_2)}$.  Applying Proposition \ref{necessary_wf} and Remark \ref{Prop 4 suau}, we get a contradiction. \\
(ii) Assume $z_2=\varphi(z_1)\neq z_1$ and let $z_0=\lim_n\varphi^n(z_1)$. Then,  $\lim_n\frac{\prod_{m=0}^{n-1}w(\varphi^m(z_1))f\circ\varphi^n(z_1)}{\prod_{m=0}^{n-1}w(\varphi^m(z_2))f\circ\varphi^n(z_2)}=\frac{w(z_1)}{w(z_0)}$. Applying Proposition \ref{necessary_wf} and Remark \ref{Prop 4 suau},  we get that $\W$ is not $\tau_p$- supercyclic. \\
(iii) From the hypothesis we can get a periodic point $z_1\in X$   of $\varphi$ and consider $z_2=\varphi(z_1).$  By Remark \ref{wfno0} there is $j\in \N$ such that $f(\varphi^j(z_1))\neq 0$ and also there exists $C_1=\max_{z\in X}|f(z)|\max_{z\in X}|w(z)|$ and $C_2=\min_{\{n\in \N: f(\varphi^n(z_1))\neq 0\}}|f(\varphi^n(z_1))|>0$ satisfying:
$$\left|\frac{\prod_{m=0}^{n-1}w(\varphi^m(z_2))f(\varphi^n(z_2))}{\prod_{m=0}^{n-1}w(\varphi^m(z_1))f(\varphi^n(z_1))}\right|=\left|\frac{w(\varphi^{n}(z_1))f(\varphi^{n+1}(z_1))}{w(z_1)f(\varphi^{n}(z_1))}\right|\leq\frac{C_1}{C_2|w(z_1)|}$$
for all $n\in \N$ such that the quotients are well defined.
Applying again Proposition \ref{necessary_wf} and Remark \ref{Prop 4 suau},  we get the contradiction\\
(iv) By Remark \ref{wfno0}, $f(z_2)\neq 0.$ Also from $\delta_{z_2}\neq 0$ in $E'$ we get $w(z_2)\neq 0$.  Let $z_1\in X$  be as in the hypothesis. Given $M=\max\{|f(z)|:z\in X\},$ we have
	$$\left | \frac{\prod_{m=0}^{n-1}w(\varphi^m(z_1))f(\varphi^{n}(z_1))}{ w(z_2)^nf(z_2)} \right |\leq \frac{M}{|f(z_2)|},\ \mbox{for all}\  n\in\N,  z\in X.$$
Applying Proposition \ref{necessary_wf} and Remark \ref{Prop 4 suau}   we get that $\W$ is not $\tau_p$- supercyclic.\\
(v)  By hypothesis and Remark \ref{wfno0}, there exists a compact neighbourhood $V$ of $z_0$ such that $\varphi(V)\subseteq V$ and $f(v)\neq 0$ for all $v\in V.$ Let $z_1\in V\backslash\{z_0\}$ and assume $z_2=\varphi(z_1)\neq z_1$ (otherwise, apply (i)).  It follows that	
$$\left | \frac{\prod_{m=0}^{n-1}w(\varphi^m(z_1))f({\varphi}^{n}(z_1))}{\prod_{m=0}^{n-1}w(\varphi^m(z_2))f({\varphi}^{n}(z_2))} \right |= 	\left | \frac{w(z_1)f({\varphi}^{n}(z_1))}{w(\varphi^n(z_1))f({\varphi}^{n+1}(z_1))} \right| \leq  \frac{M_1M_2}{M_3},\ \mbox{for all}\  n\in\N,$$
where  $M_1=\frac{\max\{|w(z)|:z\in X\}}{\min\{|w(z)|:z\in V\}},$ $M_2=\max\{|f(z)|:z\in V\}$ and $M_3=\min\{|f(v)|:v\in V\}$. From Proposition \ref{necessary_wf}, $\W$ cannot be $\tau_p$-supercyclic.
\end{proof}

\begin{remark}
 As a consequence of (ii) or (v),  for $E=H(\D)$ the weighted composition operator is not $\tau_p$-supercyclic if $\varphi$ has a fixed point in $\D.$  Even more, if $\varphi:\D\to\D$ is holomorphic and has a fixed point then $\W$ is not pointwise supercyclic on $(C(\D),\tau_p)$.
\end{remark}



In  \cite[Corollary 10]{Moradietal2017} the authors show that the composition operator on the  spaces $A(\D)$ and $ Lip_{\alpha}(\D)$ is never weakly supercyclic. In the next theorem we prove that $\tau$-supercyclicity on these spaces is not possible even for a  general weighted composition $\W$ and with respect to weaker topologies.

\begin{theorem}\label{disk}
If $X=\overline{\D}$ and $\varphi \in A(\D),$ then the weighted composition operator $\W: E\rightarrow E$ is never $\tau_p$-supercyclic.   As a consequence, $\W$ is never $\tau_p$-supercyclic on the disc algebra $A(\D)$ and the analytic Lipschitz spaces $ Lip_{\alpha}(\D),$ $0<\alpha\leq1.$
\end{theorem}

\begin{proof}
By the Denjoy-Wolff theorem \cite[Theorem 0.2]{bourdon_shapiro}, if $\varphi$ is not the identity and not an automorphism with exactly one fixed point, then there is a unique (fixed) point $z_0\in \overline{\D}$  such that $(\varphi^n)_n$ converges to $z_0$  uniformly on the compact subsets of $\D.$ Theorem \ref{noweakcases} (ii) and (v) cover all the possible cases for $\varphi,$ therefore  $\W$ is never  $\tau_p$-supercyclic.
\end{proof}

\section{Weak supercyclicity on Fr\'echet spaces}\label{sec4}

Ansari and Bourdon proved in \cite{AnsariBourdon} that if $X$ is a Banach space and $T: X\to X$ is a power bounded and supercyclic operator, then $(T^n(x))_n$ converges to 0 for each $x\in X$. From this result it follows that isometries in Banach spaces are never supercyclic. Albanese and Jornet \cite{AlbaneseJornet2018} have recently extended this result for operators in locally convex spaces. Concretely, they show  that if $E$ is a locally convex space and  $T:E\to E$ is a supercyclic operator such that $(T^n)_n$ is an equicontinuous sequence in $L(E),$ then $(T^n(e))_n$ converges  to $0$ for any $e\in E$. In particular, this property applies to barrelled spaces, where the condition of equicontinuity of   $(T^n)_n$ is equivalent to the boundedness of the sequence $(T^n(e))_n$ in $E$ for any $e\in E$. As an application of Theorem \ref{noweakcases} we get below Ansari-Bourdon type results for weakly supercyclic operators.


\begin{theorem}\label{adjointlcs}
Let $E$ be a locally convex space and let $T:E\to E$ be a continuous linear weakly supercyclic operator. If $u\in E'$ is not a fixed point of  $T'$ and it satisfies that $({T'}^n(u))_n$ is weakly convergent to some (fixed point) $v\in E',$ then  $v=0$.
\end{theorem}

\begin{proof}
Consider in $E'$  the weak-star topology $\omega^*.$  Observe that  $(E,\omega)$ can be identified as  a subspace of $(C(E'),\tau_p)$ by means of  $e(u)=\langle u,e\rangle$ for every $e\in E$ and $u\in E'.$ With this identification, $T=C_{\varphi}$ for $\varphi=T',$ as $ C_{T'}(e)(u)=(e\circ T')(u)= \langle e, T'u  \rangle= \langle Te, u\rangle$ for every $u\in E'.$ Now the result is a direct consequence of Theorem \ref{noweakcases} (ii). Observe that in the case $u$ and $T'(u)$ are not linearly independent, if $({T'}^n(u))_n=(\lambda^n u)_n,$ $\lambda\in \C,$ is weakly convergent to some $v$, then $v=0.$
\end{proof}

Assuming some  equicontinuity conditions on the operator with respect to the strong topology, the last theorem can be improved:

\begin{theorem}
\label{lcs}
Let $E$ be a locally convex space and $T:E\to E$  a continuous linear operator  which is weakly supercyclic and satisfies $q\circ T\leq q$  for a continuous norm $q$ of $E$. 
Then $\sigma_p(T)\cap \partial\D=\emptyset$ and  $\sigma_p(T')\cap \partial \D=\emptyset$. In particular, neither $T$ nor $T'$ have non zero fixed points.
	\end{theorem}

\begin{proof}
 Since $T$ is  weakly supercyclic if and only if $\alpha T$ is so for any $\alpha\in\partial \D,$ we only need to show $1\notin \sigma_p(T)$ and $1\notin \sigma_p(T')$.  Let $U=\{e\in E: \ q(e)\leq 1\}$ and consider  $K:=(U^\circ,\omega^*)$, which is a compact space by the Alaoglu Bourbaki theorem.  Notice that
 $U^{\circ}:=\{u\in E': |u(e)|\leq 1 \mbox{ for all } e\in U\}=\{u\in E': |u(e)|\leq q(e)\mbox{ for all }e\in E\}$ and from the hypothesis $q\circ T\leq q$ it follows that $T'(U^\circ)\subseteq U^{\circ}$.  There is a continuous injection $i:\ (E,\omega) \hookrightarrow (C(K),\tau_p)$. Under identification of $E$ with the corresponding subspace of $(C(K),\tau_p)$ the
operator $T$ is the composition operator $C_{\varphi}$ where $\varphi=T'$.

%
%

 The assertion  $1\notin \sigma_p(T')$, which is equivalent to saying that $\varphi$ does not have any fixed point, follows now from Theorem \ref{noweakcases} (iv). Let now see that $1\notin \sigma_p(T).$ Assume that  $T(e_0)=e_0$ for some $e_0\in U,$ $q(e_0)=1,$ and let $F(e_0):=\{u\in U^{\circ}:\  u(e_0)=1\}$. Observe that $T'(F(e_0))\subseteq F(e_0)$ and from the Hahn Banach theorem, $F(e_0)$ is nonempty. Moreover, it is a $\omega^*$-compact convex set. From Schauder-Tychonoff's fixed point theorem \cite[Page 230]{Kothe}, we get a fixed point for $T'$ and we conclude.
\end{proof}

From Theorem \ref{lcs} and the fact that an operator $T$ is weakly supercyclic if and only if $aT$ is so for each $a\in \C,$ it  easily follows the next consequence:

\begin{corollary}
\label{finalcorollary}
Let $X$ be a Banach space. If $T:X\to X$ is a weakly supercyclic operator, then  $\sigma_p(T) \subseteq B(0, \|T\|)$ and $\sigma_p(T^*) \subseteq B(0, \|T\|),$ where  $B(0, \|T\|)$ stands for the open disc of radius $\|T\|$ centered at zero.\end{corollary}

We remark that the assertion for $T^*$ in the above  corollary can also be obtained from \cite[Proposition 1.26]{BayartMatheron}. In fact, if $\alpha \in \sigma_p(T^*)$ then from \cite[Proposition 1.26]{BayartMatheron} it follows that $\sigma_p(T^*)=\{\alpha\}$ and $(1/\alpha)T$ restricted to an invariant closed hyperplane $X_0$ of $X$ is weakly hypercyclic. Hence, $|\alpha|<\|T|_{X_0}\|\leq \|T\|=\|T^*\|$.
For the special case of operators of the form $\lambda I\oplus T:\C \oplus X\to \C\oplus X$, weak supercyclicity implies weak hypercyclicity of $(1/\lambda)T$ in $X$ by  \cite[Theorem 2.2]{Sanders 2004}, and hence  also the inequality $|\lambda|<\|T\|\leq \|\lambda I\oplus T\|$.

We can now give a corollary related to \cite[Theorem 3.1]{Bes2014}. Given a simply connected domain $U,$ the operator $C_\varphi$ is hypercyclic on $H(U)$ if and only if it is weakly supercyclic. Moreover, it is equivalent to the absence of fixed points  for $\varphi$  and its  injectivity. In case $U=\D$ this is equivalent by the Denjoy-Wolff theorem to $\varphi$ being {\em strongly runaway}, that is, for each $K\subset \D$ compact there is $n_0$ such that $\varphi^n(K)\cap K=\emptyset$ for each $n\geq n_0$.  This equivalence is extended by Kalmes \cite[Theorem 23]{Kalmes2017} to composition operators on certain sheafs  $C^\infty_P(X)$ with $X\subseteq \R^d$ open and homeomorphic to $\R^d$.  Our next result involves hyperbolic domains. For the definition of a hyperbolic domain $U\subseteq \C$ we refer to  \cite{Milnor}, where it is given in a more general context. This notion includes domains which are not simply connected, as the annulus $A_r:=\{z\in \C: r<|z|<1\}$ or the punctured disc $\D\setminus\{0\}$.

\begin{corollary}
\label{runaway}
Let $U\subseteq \C$ be a hyperbolic domain and let $\varphi:U\to U$ be holomorphic. If $C_\varphi: H(U)\to H(U)$ is weakly supercyclic  then $\varphi$ is injective and  strongly runaway.
 \end{corollary}

\begin{proof}
The injectivity  of the symbol $\varphi$ was obtained by B\`es in \cite[Proposition 2.1]{Bes2014}. Suppose that $\varphi$ is not strongly runaway. By \cite[Theorem 5.2]{Milnor}, either there exists $n_0$ such that $\varphi^{n_0}=Id_U$, in which case $C_\varphi$ is certainly not weakly supercyclic, or there exists a compact subset $K$ containing an accumulation point such that $\varphi(K)\subseteq K$, and then $C_\varphi$ is not weakly supercyclic  applying Theorem \ref{lcs}, since the constant functions are fixed points of $C_{\varphi}$ and $p_K(f)=\{\sup|f(z)|:\ z\in K\}$, $f\in H(U),$ is a continuous norm in $H(U)$.
 \end{proof}

We finish this section with an application of our results strongly connected with the open question proposed in  \cite[Problem 3]{Bes2014}. 
	 For $U$ being the punctured disc or the punctured plane, we succed in showing that every composition operator on $H(U)$ is never weakly supercyclic.

\begin{theorem}
The spaces $H(\D\setminus\{0\}))$ and $H(\C\setminus\{0\})$ admit no weakly supercyclic composition operators.
\end{theorem}

\begin{proof}
We prove first the case of the punctured disc. If $\varphi$ is a self-map on $\D\setminus\{0\}$ then  $\varphi$ admits a holomorphic extension $\hat{\varphi}:\D\to \D$. If $\hat{\varphi}(0)\neq 0$, then $f\circ \varphi$ admits a holomorphic extension to $\{0\}$ for each $f\in H(\D\setminus\{0\})$.    Since $H(\D)$ is closed in $H(\D\setminus\{0\})$, we get that  $C_{\varphi}$ is not  weakly supercyclic.

Assume now that $\widehat{\varphi}(0)=0$ and that $C_\varphi$ is weakly supercyclic. By Corollary \ref{runaway}, $\varphi$ is strongly runaway, and thus, $\widehat{\varphi}$ is not an elliptic automorphism. Also by Corollary \ref{runaway} we deduce that $\widehat{\varphi}$ is  injective, and so, $\widehat{\varphi}'(z)\neq 0$ for every $z\in \D.$ Then, by Koenig's theorem \cite[Theorem 8.2]{Milnor} there is a subdomain $U$ of $\D$ such that $\widehat{\varphi}(U)\subseteq U$ and there exists $F:U\to\D$ a conformal mapping and $a\in \C,$ $0<|a|<1$ such that 
\begin{equation}
\label{konigs}
\widehat{\varphi} |_U=F^{-1}\circ g_a\circ F,
\end{equation}
\noindent where $g_a(z)=az$ for $z\in\D$. Since the continuous  restriction $H(\D\setminus\{0\})\to H(U\setminus\{0\})$, $f\mapsto f|_U$, has dense range (the functions $f_k(z)=z^{k}$,  $k\in \Z,$ span a dense subspace in both spaces by Laurent's theorem) from (\ref{konigs}) we only need to show that $C_{g_a}$ is not weakly supercyclic on $H(U\setminus\{0\})$. 
Observe that the projections $(P_k)_{k\in\Z}$ on the Laurent development are continuous functionals, and 
$$P_k(f\circ g_a)=\frac{1}{2\pi i}\int_{C_r} z^{-k-1}f(az)dz=a^{k}P_k(f), \ 0<r<1,$$

\noindent for each $k\in\Z$. Hence,  if $f$ is a weakly supercyclic vector then $P_k(f)\neq 0$ for each $k\in\Z$. We assume without loss of generality $P_0(f)=1$. If $\lim_i \lambda_i (f\circ (g_a)^{n_i})=1$,  $n_i\geq 1$ for all i,

\noindent  we get 

$$\lim_i \lambda_i=\lim_i P_0 (\lambda_i(f\circ (g_a)^{n_i}))=P_0(1)=1.$$ 

\noindent   But we also have 

$$0=|P_{-1}(1)|=\lim_i |P_{-1}(\lambda_i (f\circ (g_a)^{n_i}))|\geq \lim_i |\lambda_i| |P_{-1}(f)| \liminf_i |a^{-n_i}|\geq  \left|\frac{P_{-1}(f)}{a}\right|,$$

\noindent a contradiction.

Assume now that $C_\varphi:H(\C\setminus\{0\})\to H(\C\setminus\{0\})$ is weakly supercyclic. Then $\varphi$ is an injective holomorphic self map on $\C\setminus\{0\}$. An injective self map $\varphi$ on $\C\setminus\{0\}$ has the form $\varphi(z)=az$ or $\varphi(z)=\frac a z$,  with $a\in \C\setminus\{0\}$  (see \cite[Theorem 25.3.1]{garling}). The case  $\varphi(z)=\frac a z$ follows immediatly since $C_\varphi^2=Id.$ We then consider  the case $\varphi(z)=az$. If $|a|<1$ we proceed as in the punctured disc. If $|a|=1$ we have a rotation and the result follows directly from Corollary \ref{runaway}. If $|a|>1$ then we proceed as in the punctured disc but getting a contradiction with the projection $P_1(f)$.
\end{proof}

\section{Remarks on the dynamics of $\W$ with respect to the pointwise topology}\label{sec2}

In this section, we focus on the following problem: given a topological Haussdorff space $X$, can $\W:C(X)\rightarrow C(X)$ be pointwise supercyclic?
We answer this question in the positive for certain sequence spaces. Moreover, we study the connections between the concepts of weak supercyclicity, pointwise supercyclicity and cyclicity. It is clear that weak supercyclicity implies $\tau_p$-supercyclicity and cyclicity, since cyclicity in the weak topology is equivalent to cyclicity with respect to the norm topology. However, we show that there exist $\tau_p$-supercyclic operators which are not weakly supercyclic, and that the concepts of $\tau_p$-supercyclicity and cyclicity are not related. We illustrate these concepts and connections with several examples.

Given $X$ a locally compact space, we denote by $\hat{X}:=X\cup \{\infty\}$ the \emph{Alexandroff compactification} of $X.$  Any homeomorphism $\varphi: X\rightarrow X$ extends uniquely to $\hat{\varphi}:\hat{X}\rightarrow \hat{X}, $  $\hat{\varphi}(\infty):=\infty$ and $\hat{\varphi}_{|X}=\varphi.$

For $X=\Z,$  $C(\Z)$ is the set of bilateral sequences $\{f=(f_n)_{n=-\infty}^{\infty}: f_n\in \C, \ n\in \Z \}.$  Observe that $C(\hat{\Z})$
is the Banach space
$$c_{\infty}(\Z)=\{f=(f_n)_{n=-\infty}^{\infty},  \ f_n\in \C, \ n\in \Z: \ f_n \ \mbox{is}\ \mbox{convergent}\ \mbox{as}\ |n|\rightarrow \infty\},$$
 endowed with the supremum norm $\|f\|_{\infty}=\sup_{n\in\Z}|f_n|.$ Indeed, a bilateral sequence is a continuous map  $f: \Z\rightarrow \C,$ and it converges if and only if this map has an extension to a continuous map $\hat{f}:\hat{\Z}\rightarrow \C,$  where the basic-open neighbourhoods of $\infty$ are cofinite. The value at infinity is the limit of the sequence.

Consider now a particular case of composition operators. For each $j\in \Z,$ let $e_j$ denote the bilateral sequence $ (..., 0, 1, 0,... )$   with  the  $1$  in  the  $j$-th  position and consider  the  bilateral backward shift  $B: c_{\infty}(\Z) \rightarrow c_{\infty}(\Z)$  defined  by  $Be_j = e_{j-1}$ for  each  $j \in \Z.$ $B$ is the composition operator $C_{\varphi}:c_{\infty}(\Z) \rightarrow c_{\infty}(\Z)$ associated to the symbol $\varphi: \Z\rightarrow \Z, \ j\mapsto j+1.$     It is  well known that $B: c_{0}(\Z) \rightarrow c_{0}(\Z)$  is a weakly supercyclic isometry \cite[Theorem 2]{Sanders 2005}, where $$c_0(\Z)=\{f=(f_n)_{n=-\infty}^{\infty}\in c_{\infty}(\Z):  \ \lim_{|n|\rightarrow \infty}f_n=0\}.$$

 However, in the next example we show that  $B: c_{\infty}(\Z)\rightarrow c_{\infty}(\Z)$ is not weakly supercyclic, but it is $\tau_p$-supercyclic. As a consequence, we get that these weak forms of supercyclicity are not equivalent for composition operators.

\begin{example}\label{weakandpointwisedifferent}
 The  bilateral backward shift $B: c_{\infty}(\Z)\rightarrow c_{\infty}(\Z)$ is not weakly supercyclic but it is $\tau_p$-supercyclic. Indeed, as $B$  is weakly supercyclic on $c_0(\Z)$  \cite[Theorem 2]{Sanders 2005}, it is also $\tau_p$-supercyclic. Since $c_0(\Z)$ is dense in $c_{\infty}(\Z)$ with respect to the pointwise convergence topology on $\Z,$ we get that $B$ is  $\tau_p$-supercyclic on $c_{\infty}(\Z).$  However, $B$ is not weakly supercyclic on $c_{\infty}(\Z)=C(\hat{\Z})$  by Theorem \ref{CXcompact}, since $\hat{\Z}$ is compact and $B$ is a composition operator associated to the symbol $\varphi: \Z\rightarrow \Z, \ j\mapsto j+1.$

\end{example}

In the next proposition we show that $\tau_p$-supercyclicity does not imply cyclicity for weighted composition operators. Moreover, we prove that cyclicity together with $\tau_p$-supercyclicity do not imply weak supercyclicity. First, consider the spaces $$\ell_{\infty}=\{f=(f_n)_{n\in \N}, \ f_n\in \C, \ n\in \N:  \sup_{n\in\N}|f_n|<\infty\},$$
$$c_{\infty}=\{f=(f_n)_{n\in \N}\in \ell_{\infty} :\ \ \mbox{is}\ \mbox{convergent}\ \mbox{as}\ n\rightarrow \infty\}$$ and
$$c_0=\{f=(f_n)_{n\in \N}\in c_{\infty}: \ \lim_{n\rightarrow \infty}f_n=0\}.$$

Observe that  $c_{\infty}=C(\hat{\N}),$ where $\hat{\N}$ is the Alexandroff compactification of $\N$  and $c_0=\text{Ker} (\delta_{\infty}),$    $\delta_{\infty}: C(\hat{\N})\rightarrow \C, f\mapsto f(\infty).$ Thus, $c_0$ has codimension 1 in  $c_{\infty}.$

The unilateral weighted backward shift $B_w:  \ell_{\infty}\rightarrow \ell_{\infty},$ $w=(w_n)_{n\in\N},$ $w_n>0,$  $\lim_nw_n\rightarrow 0$ defined by $B_w e_j=w_je_{j-1},$ $j=1,2,\dots,$ $B_w e_0=e_0,$ where $e_0$ is the zero sequence,
can be viewed as a weighted composition operator with symbol  $\varphi: \N\rightarrow \N, \ j\mapsto j+1,$ and weight $w=(w_n)_{n\in\N}.$  By \cite{HildenWallen} we know that $B_w:  c_{0}\rightarrow c_{0}$  is supercyclic.

\begin{lemma}\label{codimension}
	Consider a Banach space  $G$ and a normed closed subspace $G_0\subseteq  G,$ dense with respect to the pointwise convergence topology $\tau_p.$ If the codimension of $G_0$ in $G$ is greater than 1 and $T:G_0\rightarrow G_0$ admits a continuous extension $\hat{T}:G\rightarrow G$ such that $ \hat{T}(G)\subseteq G_0,$ then $\hat{T}$ is not cyclic.
\end{lemma}

\begin{proof}
Let $f\in G$ be a cyclic vector of $\hat{T}.$ Since the codimension of $G_0$ in $G$ is greater than $1$, there exists $g\in G\setminus (\text{span}\{f\}\oplus G_0),$ and therefore, $g\notin \overline{\text{span}\{T^nf: f\geq 0\}} \subseteq\text{span}\{f\}\oplus G_0,$ a contradiction.
\end{proof}

\begin{proposition}\label{taup_nocyclic}
Consider the weight $w=(w_n)_n$ such that $w_n>0$ and  $\lim_nw_n\rightarrow 0.$
\begin{itemize}
	\item[(i)] If $f\in \ell_{\infty}\setminus c_{\infty},$ the weighted backward shift $B_w:\text{span}\{f\}\oplus c_{\infty}\rightarrow \text{span}\{f\}\oplus c_{\infty}$ is $\tau_p$-supercyclic but not (weakly super)cyclic.
	\item[(ii)]  If $f$ is a cyclic vector of $B_w: c_0\rightarrow c_0$ and there exists $g\in c_{\infty}$ such that $B_wg=f,$ then $B_w:c_{\infty}\rightarrow c_{\infty}$ is cyclic, $\tau_p$-supercyclic but not weakly supercyclic.
\end{itemize}
\end{proposition}

\begin{proof}
 By \cite{HildenWallen} it follows that $B_w:  c_{0}\rightarrow c_{0}$  is supercyclic. Thus, since $c_0$ is dense in $\ell_{\infty}$  with respect to the pointwise convergence topology, $B_w$ is  $\tau_p$-supercyclic on $\ell_{\infty},$ on $\text{span}\{f\}\oplus c_{\infty}$ and on $c_{\infty}$.\\
(i)	Since $c_0$ has codimension 1 in $c_{\infty},$ $c_0$ has codimension 2 in $\text{span}\{f\}\oplus c_{\infty}.$  Moreover, as $\lim_nw_n\rightarrow 0,$ $B_w(\text{span}\{f\}\oplus c_{\infty})\subseteq c_0$   and thus, $B_w: \text{span}\{f\}\oplus c_{\infty}\rightarrow \text{span}\{f\} \oplus c_{\infty}$ is not cyclic   by Lemma \ref{codimension}. \\
(ii) If $f\in c_0$ is a cyclic vector of $B_w:  c_0\rightarrow c_0$ and $g\in c_{\infty}\setminus c_0$ is such that $B_wg=f,$ then $c_{\infty}=\text{span}\{g\}\oplus c_0$ and $c_0=\overline{\text{span}\{B_w^ng: n\geq 1\}},$ and so, $B_w:c_{\infty}\rightarrow c_{\infty}$ is cyclic. On the other hand, since $B_w$ is a weighted composition operator and $c_{\infty}$ is the space $C(\hat{\N}),$ where  $\hat{\N}$ is the Alexandroff compactification of $\N,$ it cannot be weakly supercyclic by Theorem \ref{CXcompact}.
\end{proof}

Finally, the next example shows that cyclicity does not imply $\tau_p$-supercyclicity.

\begin{example}\label{multiplicationex}
	The multiplication operator $M_zf=zf$ is cyclic but it is not $\tau_p$-supercyclic on the disc algebra $A(\D)$. Indeed, span$\{\text{Orb}(M_z,1)\}=\text{span}\{1,z,z^2,\dots\}$ is dense in $A(\D),$ but the operator is not weakly supercyclic by \cite[Proposition 1.26]{BayartMatheron} because $\{\delta_z, \ z\in \D\}$  is a set of  independent eigenvectors of $M_z^*$ in $A(\D)^*$.
\end{example}

The next diagram  illustrates the relations between weak supercyclicity, $\tau_p$-supercyclicity and cyclicity for operators.

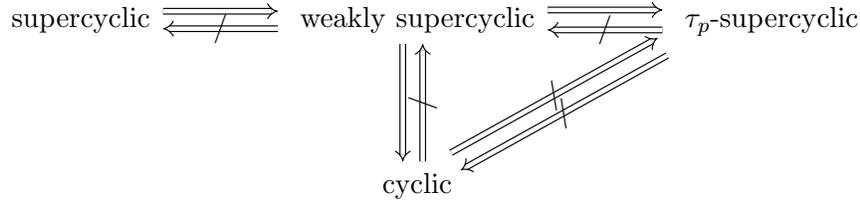
\begin{figure}[h]\label{graf}
	{\small
\begin{center}
	\begin{tikzcd}[column sep=4em,row sep=4em,inner xsep=0pt]
		\text{ supercyclic} {\arrow[r, Rightarrow,  shift left=0.7ex] \arrow[r, Leftarrow, shift left=-0.7ex, negated]} &
		\text{ weakly supercyclic} \arrow[d,Rightarrow,shift left=-0.8ex]\arrow[d,Leftarrow,shift left=0.8ex, negated] \arrow[r,Rightarrow,  shift left=0.8ex] \arrow[r,Leftarrow, shift left=-0.8ex, negated] &
		\text{ $\tau_p$-supercyclic} \arrow[dl,Rightarrow,start anchor=south west, shift left=0.8ex, negated] \arrow[dl,Leftarrow,start anchor=south west, shift left=-0.8ex, negated] \\
		& \text{ cyclic} 
	\end{tikzcd}
\end{center}} 
\end{figure}

We now analyze $\tau_p$-supercyclicity of $\W$ when $X$ is a perfect space.  We first consider $X=\overline{\D}$. Our first result follows as an inmediate consequence of Theorem \ref{noweakcases} (v):


\begin{corollary}\label{rotacionX}
	Let $\lambda\in\mathbb{C}$ with $|\lambda|\leq 1$ and let $\varphi(z)=\lambda z$, then $\W:C(\overline{\D})\rightarrow C(\overline{\D})$ is not $\tau_p$-supercyclic.
\end{corollary}
We now consider a special class of weighted compositions operators and we prove that
there are no $\tau_p$-supercyclic isometries in the space $C(\overline{\D}).$
\begin{theorem}\label{isometryonD}
If  $T:C(\overline{\D})\rightarrow C(\overline{\D})$ is a surjective isometry, then it cannot be $\tau_p$-supercyclic.
\end{theorem}

\begin{proof}
Let  $T:C(\overline{\D})\rightarrow C(\overline{\D})$ be  a surjective isometry. From the Banach-Stone Theorem \cite{stone},  there exist a homeomorphism $\varphi:\overline{\D}\rightarrow \overline{\D}$ and $w:\overline{\D}\rightarrow\C$ such that $|w(z)|=1$ for all $z\in \overline{\D}$ and $T=\W.$ Since $\overline{\D}$ is convex, by the Brouwer fixed-point theorem, $\varphi$ has a fixed point $z_0\in \overline{\D}.$ Suppose that $\W$ is $\tau_p$- supercyclic and let $f$ be a $\tau_p$-supercyclic function of $\W$. By Remark \ref{wfno0}, it follows that $f(z_0)\neq0.$ Let $z_1\in \overline{\D}$ with $z_1\neq z_0$ and $M=\max\{f(z):z\in \overline{\D}\},$ then
	$$\left | \frac{\prod_{m=0}^{n-1}w(\varphi^m(z_1))f({\varphi}^{n}(z_1))}{w(z_0)^nf(z_0)} \right |\leq \frac{M}{|f(z_0)|},\ \mbox{for all}\  n\in\N.$$	
	From Proposition \ref{necessary_wf}, it follows that $T$ cannot be $\tau_p$-supercyclic.
\end{proof}


We now analyze  $\tau_p$-supercyclicity of weighted composition operators acting on  $C(\partial\D),$ the space of continuous functions on the unit circle. By Proposition \ref{necessaryCK}(ii) and the next Lemma it is enough to consider that the symbol $\varphi:\partial\D\rightarrow\partial\D$ is a homeomorphism.

\begin{lemma}\label{inyectividadtoro}
	If $\varphi:\partial\D\rightarrow\partial\D$ is  continuous and injective, then $\varphi$ is a homeomorphism.
\end{lemma}
\begin{proof}
	Suppose there exists $\varphi:\partial\D\rightarrow\partial\D$   continuous and injective  such that $\varphi(\partial\D)\neq \partial\D.$ Since $\varphi$ is continuous, $\varphi(\partial\D)$ is compact and connected. Let $t$ be an interior point of $ \varphi(\partial\D)$ and consider $z\in \partial\D$ such that  $t=\varphi(z).$ Since $\varphi$ is injective and no surjective, $\varphi(\partial\D\backslash\{z\})=\varphi(\partial\D)\setminus\{t\}$ is not connected, but $\partial\D\backslash\{z\}$ is so, a contradiction.
\end{proof}

In the next propositions we show that $\W$ is not $\tau_p$-supercyclic on $C(\partial\D)$ when $\varphi$ has periodic points or it is a rotation and $|w|=1.$


\begin{proposition}\label{discoperiodicos}
	Let $\varphi:\partial\D\rightarrow\partial\D$ be a homeomorphism with a periodic point. The weighted composition operator  $\W:C(\partial\D)\rightarrow C(\partial\D)$ is not $\tau_p$-supercyclic.
\end{proposition}
\begin{proof}
	The dynamics in the case with just a fixed point $z_0\in \partial\D$ reduces to the dynamics of functions $f$ on the interval $[0,1]$  that are invertible, with $f(0)=0,$ $f(1)=1$ and $f(x)\neq x$ for every $x\in (0,1)$ \cite[Section 2]{gadgil}.  Since $f(x)-x$ has constant sign, we can suppose w.l.o.g. that $f(x)-x>0,$ and then it follows that $x<\dots<f^{n-1}(x)<f^n(x)<1$ for all $x\in(0,1).$ As a consequence, we get that the iterates must converge to a fixed point, that is, $\lim_{n\rightarrow \infty}f^n(x)=1,$ and then $\varphi^n(z)$ converges to $z_0$ for every $z\in\partial\D\setminus\{z_0\}.$ By Theorem \ref{noweakcases} (ii),  $\W$ is not $\tau_p$-supercyclic. In the case $\varphi$ has more than one fixed point, or has a periodic point, Theorem \ref{noweakcases} (i) and (iii) yield the result.
\end{proof}

\begin{proposition}	\label{notauprotacio}
Let $\varphi: \partial\D\rightarrow \partial\D,$  $\varphi(z)=\lambda z$ for some $\lambda \in \C$ with $|\lambda|=1$ and $w$ such that $|w|=1.$ The operator $\W:C(\partial\D)\rightarrow C(\partial\D)$ is never $\tau_p$-supercyclic.
\end{proposition}

\begin{proof}
Assume that $\W$ is $\tau_p$-supercyclic. Let $f$ be a  $\tau_p$-supercyclic function such that $\|f\|_{\infty}=1.$ Given $g\in C(\partial\D)$ such that $\|g\|_{\infty}=1,$ there exists a net  such that $\left(\alpha_iC_{\varphi,w}^{n(i)}f\right)_{i\in I}\rightarrow g$ in $\tau_p,$ and thus, $\left(|\alpha_i||f(\varphi^{n(i)})| \right)_{i\in I}\rightarrow |g|$ in $\tau_p.$ Since $\partial\D$ is compact, by extracting a subnet, we can assume w.l.o.g. that $\lambda^{n(i)}\rightarrow \lambda_0$ for some $\lambda_0\in \partial\D.$ Let $z_0\in \partial\D$ such that $|f(\lambda_0 z_0)|=1.$ Since $|f(\varphi^{n(i)}(z_0))|\rightarrow |f(\lambda_0 z_0)|=1$ and $|f(\varphi^{n(i)}(z_0))|\leq 1,$ there exists $i_0$ such that $|f(\varphi^{n(i)}(z_0))|\in [1/2, 1]$ and $$ |\alpha_i||f(\varphi^{n(i)}(z_0))|-|g(z_0)|<1,\  \mbox{i.e.},\  |\alpha_i||f(\varphi^{n(i)}(z_0))|< 2$$ for every $i\geq i_0.$ Hence, $|\alpha_i|\leq \frac{2}{|f(\varphi^{n(i)}(z_0))|}\leq 4.$   Let  $\left(|\alpha_j||f(\varphi^{n(j)})| \right)_{j\in J},$ $J\subseteq I,$ be a subnet  such that $|\alpha_j|_{j\in J}$ converges to $\alpha\geq 0.$ Now, $\left(|\alpha_j||f(\varphi^{n(j)})| \right)_{j\in J}\rightarrow \alpha|f(\lambda_0 \ \cdot)|$ in $\tau_p$ and also  $\left(|\alpha_j||f(\varphi^{n(j)})| \right)_{j\in J}\rightarrow |g|.$ Hence, $\alpha=1$ and $|g(z)|=|f(\lambda_0 z)|$ for every $z\in \partial\D.$ Therefore, it follows that $\|g\|_{L^1}=\|f\|_{L^1}$ for every $g\in B_{C(\partial\D)} ,$ a contradiction.
\end{proof}

Given  an interval $I$ and $\varphi: I\rightarrow I,$ the $\omega$-limit set of the orbit of $x\in I$ is the set
$$\omega(x) = \{y\in I :\  \exists  n_k \text{ with } \lim_{k\rightarrow \infty}\varphi^{n_k}(x)=y\}.$$
We say that $\varphi: I\rightarrow I$  has a \emph{wandering interval} $J$ if the intervals $J, \varphi(J), \varphi^2(J),\dots$ are pairwise disjoint and  the $\omega$-limit set of $J$ is not equal to a single periodic orbit. 	 Examples of symbols $\varphi:\partial\D\rightarrow \partial\D$ with non-wandering intervals are the $C^2$ diffeomorphisms without periodic points, the $C^1$ diffeomorphisms whose derivative is a function of bounded variation or the analytic homeomorphisms of the circle without periodic points \cite[Chapter 1, Section 2]{wellington}.

\begin{theorem}\label{Tnotaup}
If $\varphi: \partial\D\rightarrow \partial\D$ does not have  a wandering interval $J$ and $|w|=1,$ then the weighted composition operator $\W: C(\partial\D)\rightarrow C(\partial\D)$ is not $\tau_p$-supercyclic.
\end{theorem}

\begin{proof}
Assume $\varphi: \partial\D\rightarrow \partial\D$ does not have  a wandering interval $J$ and has no periodic points (otherwise, apply Proposition \ref{discoperiodicos}). By \cite[page 36]{wellington},  $\varphi$ is conjugate to a rotation $\tilde{\varphi},$ i.e., there exists a homeomorphism $h:\partial\D\rightarrow \partial\D$ such that $\tilde{\varphi}=h^{-1}\circ\varphi\circ h,$ and thus, $\W$ is similar to $C_{w\circ h, \tilde{\varphi}} $ \cite{Gunatillake}. Since $C_{w \circ h, \tilde{\varphi}} $ is not $\tau_p$-supercyclic by Proposition \ref{notauprotacio}, then $\W$ is not $\tau_p$-supercyclic \cite{liangzhou}.
\end{proof}

From all our results it seems natural to conjecture  that if  $X$ is compact,  $E\hookrightarrow (C(X),\tau_p)$ is a Banach space and the operator $\W: E \rightarrow E$  is isometric, then $\W$  is not $\tau_p$-supercyclic.



\subsection*{Acknowledgments}
The first and the second author were supported by MEC, MTM2016-76647-P. The third author was supported by MEC, MTM2016-75963-P and GVA/2018/110.

\end{document}